\documentclass[12pt]{amsart}
\usepackage{amssymb}
\usepackage{enumerate}

\usepackage{graphicx}
\usepackage[usenames,dvipsnames,svgnames,table]{xcolor}
\usepackage[a4paper]{geometry}
\usepackage{psfrag,xmpmulti,amscd,color,pstricks, import}

\newcommand{\eps}{\varepsilon}

\newcommand{\K}{{\mathcal K}}
\newcommand{\N}{{\mathbb N}}
\newcommand{\C}{{\mathbb C}}

\newcommand{\R}{{\mathbb R}}

\newcommand{\tef}{transcendental entire function}

\newcommand\qfor{\;\;\text{for }}

\newcommand{\mt}{\widetilde{m}}

\theoremstyle{plain}
\newtheorem{theorem}{Theorem}[section]

\newtheorem*{theorem*}{Theorem}
\newtheorem*{proposition*}{Proposition}

\newtheorem{lemma}[theorem]{Lemma}
\theoremstyle{definition}
\newtheorem{definition}[theorem]{Definition}
\theoremstyle{remark}
\newtheorem*{remark*}{Remark}
\newtheorem*{remarks*}{Remarks}

\theoremstyle{problem}

\theoremstyle{example}
\newtheorem{example}[theorem]{Example}
\newtheorem*{example*}{Example}
\theoremstyle{question}

\theoremstyle{questions}
\newtheorem*{questions*}{Questions}

{\begin{list}{}%
         {\setlength{\leftmargin}{#1}}%
         \item[]%
}
{\end{list}}

\begin{document}


\title[Iterating the minimum modulus]{Iterating the minimum modulus: functions of order half, minimal type}

\author{D. A. Nicks}
\address{School of Mathematical Sciences\\
The University of Nottingham\\
University Park\\
Nottingham NG7 2RD\\
UK}
\email{Dan.Nicks@nottingham.ac.uk}

\author{P. J. Rippon}
\address{School of Mathematics and Statistics \\
The Open University \\
   Walton Hall\\
   Milton Keynes MK7 6AA\\
   UK}
\email{Phil.Rippon@open.ac.uk}

\author{G. M. Stallard}
\address{School of Mathematics and Statistics \\
The Open University \\
   Walton Hall\\
   Milton Keynes MK7 6AA\\
   UK}
\email{Gwyneth.Stallard@open.ac.uk}

\thanks{The last two authors were supported by the EPSRC grant EP/R010560/1.}



\begin{abstract}
For a {\tef}~$f$, the property that there exists $r>0$ such that $m^n(r)\to\infty$ as $n\to\infty$, where $m(r)=\min \{|f(z)|:|z|=r\}$, is related to conjectures of Eremenko and of Baker, for both of which order $1/2$ minimal type is a significant rate of growth. We show that this property holds for functions of order~$1/2$ minimal type if the maximum modulus of~$f$ has sufficiently regular growth and we give examples to show the sharpness of our results by using a recent generalisation of Kjellberg's method of constructing entire functions of small growth, which allows rather precise control of $m(r)$.
\end{abstract}
\maketitle
\begin{center}
{\it  Dedicated to the memory of Walter Hayman.}\\
\end{center}

\section{Introduction}\label{intro}
\setcounter{equation}{0}
Let $f$ be a {\tef} and denote by $f^n$, $n\in\N$, the $n$-th iterate of~$f$. The results in this paper address a question in complex dynamics concerning the {\it escaping set} of a {\tef}~$f$, defined as
\[
I(f)=\{z:f^n(z) \to \infty \;\text{ as } n\to \infty\},
\]
namely, whether all the components of $I(f)$ are unbounded. This question is known as {\it Eremenko's conjecture} \cite{E} and despite much work it remains open.

For any {\tef}~$f$, the {\it maximum modulus} and {\it minimum modulus} of~$f$ are defined as follows, for $r\ge 0$:
\[
M(r)=M(r,f)=\max_{|z|=r}|f(z)|\quad\text{and}\quad m(r)=m(r,f)=\min_{|z|=r}|f(z)|,
\]
respectively. There is a huge literature about the relationship between $m(r)$ and $M(r)$ for various types of {\tef}s. Clearly $m(r)<M(r)$ for all $r>0$ and the function $M(r)$ is strictly increasing and unbounded. On the other hand, the function $m(r)$ is  alternately increasing and decreasing between adjacent values of~$r$ for which $m(r)=0$ and eventually decreases to~0 in the case that~$f$ has only finitely many zeros. For certain functions, however, $m(r)$ is comparable in size to $M(r)$ for an unbounded set of values of~$r$.

We let $M^n(r)$ and $m^n(r)$ be defined by iterating the real functions $M(r)$ and $m(r)$ respectively. For any {\tef}~$f$ we have
\begin{equation}\label{Mnr}
M^n(r) \to \infty\;\text{ as }n\to\infty,
\end{equation}
for $r\ge R=R(f)$ say, but for the iterated minimum modulus the property:
\begin{equation}\label{minmodprop}
\text{there exists } r > 0 \text{ such that } m^n(r) \to \infty \text{ as } n \to \infty,
\end{equation}
may or may not hold, depending on the function~$f$.

It has been known for some time that the sequence $M^n(r)$ is of importance in relation to work on Eremenko's conjecture, since it plays a key role in the definition of a subset of $I(f)$ called the {\it fast escaping set}, all of whose components are unbounded; see, for example, \cite{RS10a}. More recently, it has been observed that property~\eqref{minmodprop}, when it is true, can also play an important role in relation to this conjecture; see \cite{ORS1}, \cite{ORS2} and~\cite{NRS19}. For example, in \cite{NRS19} we obtained the following result, which gives a family of {\tef}s for which Eremenko's conjecture holds in a particularly strong way.
\begin{theorem}\label{Ifconn}
Let~$f$ be a real {\tef} of finite order with only real zeros for which there exists $r>0$ such that $m^n(r) \to \infty$ as $n \to \infty.$ Then $I(f)$ is connected and hence Eremenko's conjecture holds for~$f$.
\end{theorem}
{\em Remark}\quad In \cite[Theorem~1.1]{NRS19} we showed moreover that, for such functions $f$, the set $I(f)$ has the structure of an (infinite) spider's web.

Recall that the {\it order} $\rho(f)$ and the {\it lower order} $\lambda(f)$ of a {\tef}~$f$ are
\[
\rho(f)=\limsup_{r\to\infty}\frac{\log\log M(r)}{\log r}\quad\text{and}\quad \lambda(f)=\liminf_{r\to\infty}\frac{\log\log M(r)}{\log r},
\]
respectively, and~$f$ is said to be {\it real} if $f(\bar z)=\overline{f(z)},\qfor z\in\C.$

In view of Theorem~\ref{Ifconn}, it is natural to ask which {\tef}s satisfy property \eqref{minmodprop}; in particular, which amongst those that satisfy the hypotheses of this theorem. In \cite{NRS19} we proved several results in relation to this question, including the following.
\begin{itemize}
\item[(a)]
For a real {\tef}~$f$ of finite order with only real zeros:
\begin{itemize}
\item[(i)]
if $0\le \rho(f)<1/2$, then property \eqref{minmodprop} holds;
\item[(ii)]
if $\rho(f)>2$, then property \eqref{minmodprop} does not hold.
\end{itemize}
\item[(b)]
For any $\rho\in [1/2,2]$ there are examples of real {\tef}s with only real zeros of order~$\rho$ for which property \eqref{minmodprop} holds and also examples of such functions for which \eqref{minmodprop} does not hold.
\end{itemize}

Actually, property \eqref{minmodprop} holds for {\em all} entire functions of order less than~$1/2$; see \cite[Theorem~1.1]{ORS2}. This follows from the cos $\pi \rho$ theorem which shows that for functions of order less than~$1/2$ there is a close relation between the minimum modulus and maximum modulus for many values of $r$; see \cite[Theorem~6.13]{wH89}.

For functions of order $1/2$, minimal type, weaker results on the size of the minimum modulus are known; in particular, Wiman showed that, for such functions, $m(r)$ is unbounded on $(0,\infty)$; see \cite[Theorem~6.4]{wH89}. In view of Wiman's result, one might hope that property \eqref{minmodprop} holds for real entire functions with only real zeros of order 1/2, minimal type, that is, when
\begin{equation}
\limsup_{r\to\infty}\frac{\log\log M(r)}{\log r}=\frac12\quad\text{and}\quad \lim_{r\to\infty} \frac{\log M(r)}{r^{1/2}}=0.
\end{equation}

In this paper, we show that \eqref{minmodprop} does indeed hold for many families of such functions, but does not hold for all.
Our main positive result is the following which, roughly speaking, shows that \eqref{minmodprop} holds whenever $(\log M(r))/r^{1/2}$ tends to 0 (as $r\to \infty)$ in a sufficiently regular manner.

\begin{theorem}\label{main1}
Let~$f$ be a {\tef} of order at most 1/2, minimal type, and suppose that there exists $r_0>0$ such that, for $r>r_0$,
\begin{equation}\label{main-crit}
\frac{\log M(r)}{r^{1/2}}\le \frac{1}{4}\frac{\log M(s)}{s^{1/2}},
\end{equation}
for some $s \in (0,r)$ which satisfies $M(s)\ge r^2$. Then there exists $r>0$ such that $m^n(r) \to \infty$ as $n \to \infty.$
\end{theorem}

{\it Remarks }\;1. The second hypothesis of Theorem~\ref{main1} implies the first hypothesis, that~$f$ has order at most 1/2, minimal type, though this is not quite obvious.

2. Theorem~\ref{main1} implies the result stated earlier that, if $\rho(f)<1/2$, then property~\eqref{minmodprop} holds. Indeed, if $\rho(f)<1/2$ and we put
\[
\delta:=1/2-\rho(f)\in (0,1/2],
\]
then \eqref{main-crit} holds with $s=r^{\delta}$ for~$r$ sufficiently large, since this choice of~$s$ gives $r^{1/2}/s^{1/2}= r^{\rho(f)+\delta/2}$ and $M(r^{\delta}) \ge r^2$ for~$r$ sufficiently large.

In Section 2, we give a number of conditions on the maximum modulus and on the zeros of~$f$ which imply the hypotheses of Theorem~\ref{main1} and are convenient for applications.

In the second half of the paper we use a recent generalisation of a method of Kjellberg \cite{RS20} to construct examples of functions of order $1/2$, minimal type, for which property \eqref{minmodprop} does not hold. First, we show that no matter how {\em slowly} $\log M(r)$ grows, consistent with~$f$ having order 1/2, minimal type, we cannot deduce that property \eqref{minmodprop} holds.

\begin{example}\label{main5}
Let~$\delta :(0,\infty) \to (0,\tfrac12)$ be a decreasing function such that
\[
\delta(r)\to 0\;\text{ as } r \to \infty.
\]
Then there exists a real {\tef}~$f$ of order $1/2$,  minimal type, with only real zeros, and $r_0>0$ such that
\begin{equation}\label{max}
\log M(r,f)\le r^{1/2-\delta(r)},\qfor r \ge  r_0,
\end{equation}
for which property~\eqref{minmodprop} does not hold.
\end{example}

Next we show that making the additional assumption of positive lower order, or even of lower order $1/2$, is insufficient to ensure that property \eqref{minmodprop} holds.

\begin{example}\label{main4}
There exists a real {\tef} $f$ of order $1/2$, minimal type, and of lower order $1/2$, with only real zeros such that property \eqref{minmodprop} does not hold.
\end{example}

Finally, we remark that consideration of the iterates of the minimum modulus first arose in connection with the conjecture of Baker that entire functions of order $1/2$, minimal type, have no unbounded Fatou components; see \cite{NRS18} for the most recent progress on this conjecture.

The structure of the paper is as follows. In Section~\ref{orderhalf-positive} we prove our positive results, including Theorem~\ref{main1}, and in Section~\ref{background} we recall some results from \cite{RS20} needed for the proofs of Examples~\ref{main5} and \ref{main4}, which are given in Section~\ref{orderhalf-negative}.

\section{Proof of Theorem~\ref{main1} and some special cases}
\label{orderhalf-positive}
\setcounter{equation}{0}
In this section we prove Theorem~\ref{main1}, our positive result about property~\eqref{minmodprop}, and give some special cases of it which we use to obtain examples of functions that satisfy the hypotheses of Theorem~\ref{main1}.

The proof of Theorem~\ref{main1} depends on the following lemma of Beurling~\cite[page~95]{aB33}. For any subharmonic function~$u$ we write
\[
B(r,u)=\max_{|z|=r}u(z),\;\text{where } r>0.
\]

\begin{lemma}\label{Beur}
If~$u$ is subharmonic in~$\C$, $0<r_1<r_2$, and
\[
E(r_1,r_2)=\{r\in [r_1,r_2]:\inf_{|z|=r}u(z)\le 0\},
\]
then
\begin{equation}\label{Beur-est}
B(r_2,u)> \frac12 \exp\left(\frac12\int_{E(r_1,r_2)} \frac{1}{t}\,dt\right)B(r_1,u).
\end{equation}
In particular, if $E(r_1,r_2)=[r_1,r_2]$, then
\[
B(r_2,u)\ge \frac12 \left(\frac{r_2}{r_1}\right)^{1/2} B(r_1,u).
\]
\end{lemma}
In order to work with property~\eqref{minmodprop}, the function
\[
\mt(r) := \max\{ m(s):0 \leq s \leq r\}, \qfor r \in [0, \infty),
\]
was introduced in \cite{ORS2} and it was shown that property~(\ref{minmodprop}) is true if and only if
\begin{equation}\label{mtequiv}
\text{there exists } R > 0 \text{ such that } \mt(r) >r, \text{ for } r \geq R.
\end{equation}
We use this equivalent property to prove Theorem~\ref{main1}.
\begin{proof}[Proof of Theorem~\ref{main1}]
Suppose the hypotheses of Theorem~\ref{main1} hold. If property~\eqref{minmodprop} does not hold, then by \eqref{mtequiv} there exist arbitrarily large $r>r_0$ such that $\mt(r) \le r$; that is, for arbitrarily large~$r$ we have
\begin{equation}\label{mtr}
m(t) \le r,\qfor 0<t\le r.
\end{equation}
For such an~$r$, there exists by hypothesis $s<r$ such that $\log M(s)\ge 2\log r$ and \eqref{main-crit} holds. However, we deduce from \eqref{mtr} by applying Lemma~\ref{Beur} to $u(z)=\log(|f(z)|/r)$ that
\[
\log M(r)>\log \frac{M(r)}{r} > \frac12\left(\frac{r}{s}\right)^{1/2}\log \frac{M(s)}{r}\ge \frac14\left(\frac{r}{s}\right)^{1/2}\log M(s),
\]
which contradicts condition \eqref{main-crit}.
\end{proof}

For applications of Theorem~\ref{main1}, it is often useful to express condition~\eqref{main-crit} in terms of the functions $\eps(r), r>0,$  and $k(r), r>0,$ defined as follows:
\[
\log M(r)=r^{1/2-\eps(r)}\quad \text{and}\quad k(r)=\eps(r) \log r;
\]
equivalent forms of condition~\eqref{main-crit} are then
\begin{equation}\label{k-crit1}
s^{\eps(s)}\le \frac14 r^{\eps(r)}\quad\text{and}\quad k(r)\ge k(s)+\log 4,
\end{equation}
for $r>r_0$, where $M(s)\ge r^2$ as before.

Note that if $f$ has order at most 1/2, minimal type, then for all sufficiently large values of~$r$ we have $0<\eps(r)<1/2$ and hence $0<k(r)=\eps(r)\log r<\frac12 \log r$.

It is also useful to note that if~$f$ has order 1/2 then:
\begin{itemize}
\item[(a)] $f$ has lower order 1/2 if and only if $\eps(r)\to 0$ as $r\to\infty$,
\item[(b)] $f$ is of minimal type if and only if $k(r)\to\infty$ as $r\to \infty$.
\end{itemize}

We now give two special cases of Theorem~\ref{main1} which are convenient for applications. Theorem~\ref{main2}, part~(a), shows that property \eqref{minmodprop} holds whenever the function~$f$ has positive lower order and~$\eps(r)=k(r)/\log r$ does not tend to 0 too quickly. Example~\ref{main5} shows that positive lower order alone is not sufficient here.

\begin{theorem}\label{main2}
Let $f$ be a {\tef} of order 1/2, minimal type, and let $\eps(r), r>0,$ and $k(r), r>0,$ be defined as above. Then there exists $r>0$ such that $m^n(r) \to \infty$ as $n \to \infty$ if either of the following statements holds:
\begin{itemize}
\item[(a)] there exist $\delta\in (0,1/2)$, $C>1$ and $r_1>0$ such that
\begin{equation}\label{k-crit2}
\delta \log r \ge k(r) \ge C\frac{\delta}{1/2-\delta} \log\log r, \qfor r>r_1;
\end{equation}
\item[(b)] $\eps(r) \to 0$ as $r\to \infty$ and there exist $C,d>1$ and $r_2>0$ such that
\begin{equation}\label{k-crit3}
k(r) \ge Ck((\log r)^{2d}), \qfor r>r_2.
\end{equation}
\end{itemize}
\end{theorem}


\begin{proof}
By Theorem~\ref{main1}, it is sufficient to show that, for all sufficiently large~$r$, condition \eqref{main-crit} holds for some~$s\in (0,r)$ with $M(s)\ge r^2$.

(a) The inequality $k(r)\le\delta\log r$ can be written as $\eps(r)\le \delta$, that is,
\[
\log M(r)\ge r^{1/2-\delta},
\]
so the condition $M(s)\ge r^2$ is satisfied by taking $s=(2\log r)^{1/(1/2-\delta)}$. With this choice of~$s$ the required condition \eqref{main-crit}, or equivalently \eqref{k-crit1}, can be written as
\begin{equation}\label{k-crit2}
k(r)\ge k((2\log r)^{1/(1/2-\delta)})+\log 4,\qfor r>r_0.
\end{equation}
The hypothesis $k(r)\le\delta\log r$ for $r>r_1$ implies that
\begin{align*}
k((2\log r)^{1/(1/2-\delta)})&\le\delta\log \left((2\log r)^{1/(1/2-\delta)}\right)\\
&=\frac{\delta}{1/2-\delta}(\log 2+\log\log r),\qfor (2\log r)^{1/(1/2-\delta)}>r_1,
\end{align*}
so a condition of the form
\[
k(r)\ge \frac{\delta}{1/2-\delta}(\log 2+\log\log r)+\log 4,\qfor r>r_1,
\]
is sufficient to imply that \eqref{k-crit2} and hence \eqref{main-crit} hold for some $r_0\ge r_1$. This proves part~(a).

(b) We first choose $\delta>0$ such that $1/(1/2-\delta)<2d$. The hypothesis that $\eps(r)\to 0$ as $r\to\infty$ implies that there exists $r(C,d)>0$ such that $\eps(r)\le\delta/C$ for $r>r(C,d)$, so
\[
\log M(r) \ge r^{1/2-\delta/C}, \qfor r>r(C,d).
\]
Therefore, the condition $M(s)\ge r^2$ is satisfied, for sufficiently large~$r$, by taking $s=(\log r)^{1/(1/2-\delta)}$, since $C>1$.
With this choice of~$s$ the required condition \eqref{main-crit}, or equivalently \eqref{k-crit1}, can be written as
\[
k(r)\ge k((\log r)^{1/(1/2-\delta)})+\log 4,\qfor r>r_0,
\]
so a condition of the form
\[
k(r)\ge Ck((\log r)^{1/(1/2-\delta)}),\qfor r>r_2,
\]
is sufficient to imply that \eqref{main-crit} holds for some $r_0\ge \max\{r_2,r(C,d)\}$. This proves part~(b).
\end{proof}

Theorem~\ref{main2} can be used to give many examples of {\tef}s of order 1/2, minimal type, for which property \eqref{minmodprop} holds, including ones for which we can deduce that $I(f)$ is connected by using Theorem~\ref{Ifconn}. We give here some examples of this type, derived from Theorem~\ref{main2}, part~(b), in which the functions have a very regular distribution of zeros.

For any {\tef}~$f$ we define $n(r)=n(r,f)$ to be the number of zeros of~$f$ in $\{z: |z| \leq r\}$, counted according to multiplicity.

\begin{theorem}\label{main3}
Let $\eps(r)$ satisfy the hypotheses of Theorem~\ref{main2}, part~(b) and in addition suppose that $\eps(r)$ is decreasing and $k(r)=\eps(r)\log r$ is increasing.

Let~$f$ be a {\tef} of the form
\[
f(z)=\prod_{n\in \N}\left(1-\frac{z}{a_n}\right),\quad 0<|a_1|<|a_2|< \cdots,
\]
where
\[
n(r,f)\sim r^{1/2-\eps(r)}\;\text{ as } r\to\infty.
\]
Then
\begin{itemize}
\item[(a)] there exist constants $1<A<B$ and $R>0$ such that
\[
Ar^{1/2-\eps(r)}\le \log M(r,f)\le Br^{1/2-\eps(r)},\qfor r>R,
\]
\item[(b)] there exists $r>0$ such that $m^n(r) \to \infty$ as $n \to \infty$.
\end{itemize}
\end{theorem}

{\it Remark}\; A family of functions that satisfy the hypotheses of Theorem~\ref{main2}, part~(b), and Theorem~\ref{main3} is given by
\[
k(r)=\alpha(\log^n r)^{\beta},\quad \alpha,\beta>0, n\ge 2,
\]
where $\log^n$ denotes the $n$-th iterated logarithm.

For the proof of Theorem~\ref{main3}, we require a result about the relationship between $M(r)=M(r,f)$ and the following quantities:
\[
N(r) = \int_0^r \frac{n(t)}{t}\,dt\quad\text{and}\quad Q(r) = r \int_r^{\infty} \frac{n(t)}{t^2}\, dt.
\]

We use the following estimates; see~\cite[Lemma 3.3]{RS08}, for example.

\begin{lemma}\label{NMQ}
Let~$f$ be a transcendental entire function of order less than~$1$ with $f(0) = 1$. Then, for $r > 0$,
\[
N(r) \le \log M(r) \leq N(r) + Q(r).
\]
\end{lemma}
\begin{proof}[Proof of Theorem~\ref{main3}]
Since $n(r)\sim r^{1/2-\eps(r)}$ as $r\to\infty$, where $\eps(r)$ is positive and decreasing to 0, and $k(r)=\eps(r)\log r$ is increasing, we have
\[
N(r)= \int_0^r \frac{1+o(1)}{t^{1/2+\eps(t)}}\,dt\ge \frac{(1+o(1))}{r^{\eps(r)}}\int_0^r \frac{dt}{t^{1/2}}=2 r^{1/2-\eps(r)}(1+o(1))
\]
and
\[
Q(r) = r\int_r^{\infty}\frac{1+o(1)}{t^{3/2+\eps(t)}}\,dt\le (1+o(1))r^{1-\eps(r)}\int_r^{\infty}\frac{dt}{t^{3/2}} = 2 r^{1/2-\eps(r)}(1+o(1)),
\]
so
\[
Q(r)\le N(r)(1+o(1)) \;\text{ as }r\to\infty.
\]
Since $\eps(r)$ is decreasing, we have
\[
N(r)= \int_0^r \frac{(1+o(1))}{t^{1/2+\eps(t)}}\,dt\le (1+o(1))\int_0^r \frac{dt}{t^{1/2+\eps(r)}}=\frac{1+o(1)}{1/2-\eps(r)}r^{1/2-\eps(r)},
\]
so
\[
N(r)+Q(r)\le 4r^{1/2-\eps(r)}(1+o(1)) \;\text{ as }r\to\infty.
\]
Part~(a) follows by Lemma~\ref{NMQ} and the above inequalities for $N(r)$ and $Q(r)$.

It follows from part~(a) that
\[
\log M(r,f)=r^{1/2-\hat\eps(r)},\quad\text{where}\quad \hat\eps(r)=\eps(r)+O(1)/\log r\;\text{ as }r\to\infty,
\]
so $\hat\eps(r)\to 0$ as  $r\to \infty$ and $\hat k(r)=\hat\eps(r)\log r$ satisfies
\[
\hat k(r)=k(r)+O(1)\;\text{ as }r\to\infty.
\]
Since $k(r)$ satisfies \eqref{k-crit3}, so in particular $k(r)\to\infty$ as $r\to\infty$, it follows that $\hat k(r)$ also satisfies \eqref{k-crit3}. Hence~$f$ satisfies the hypotheses of Theorem~\ref{main2}, part~(b), so the proof is complete.
\end{proof}

\section{Constructing entire functions of small order}
\label{background}
\setcounter{equation}{0}

Our method of proving Examples~\ref{main4} and~\ref{main5} uses a recent generalisation \cite{RS20} of Kjellberg's method~\cite[Chapter~2]{bK48} for constructing {\tef}s of order less than $1/2$ by approximating certain continuous subharmonic functions by functions of the form $\log |f|$ where~$f$ is a {\tef}. In this section we summarise the results from \cite{RS20} which are needed to construct our examples.

Kjellberg's method is a two stage process:
\begin{itemize}
\item[1.]
a continuous subharmonic function~$u$ with the required properties is obtained by using a positive harmonic function defined in the complement of a particular sequence of radial slits, on which~$u$ vanishes;
\item[2.]
the Riesz measure of~$u$ is discretised to produce an entire function~$f$ such that $\log |f|$ is close to~$u$ away from the zeros of~$f$.
\end{itemize}
%

The paper~\cite{RS20} gives a generalisation of Kjellberg's method which allows the slits to be chosen more flexibly than in \cite{bK48}. 

%

We now recall some of the key definitions and terminology needed to state the results from \cite{RS20}.

\begin{definition}
A subharmonic function~$u$ is in the class ${\mathcal K}$ if $u$ is continuous in $\C$ and positive harmonic in $D=\C\setminus E$, where $E\subset (-\infty,0]$ is a closed set on which $u$ vanishes. We assume that each point of~$E$ is regular for the Dirichlet problem in~$D$.
\end{definition}

{\it Remark}\; For each closed subset of the negative real axis, there is exactly one corresponding function $u\in \K$ up to positive scalar multiples, by a result of Benedicks \cite[Theorem~4]{Ben}.

Recall that for a set $S\subset \R^+$ and $r>1$, we define the {\it upper logarithmic density} of~$S$,
\[
\overline{\Lambda}(S) = \limsup_{r\to \infty}\frac{1}{\log r}\int_{S\cap(1,r)}\frac{dt}{t},
\]
and the {\it lower logarithmic density} of $S$,
\[
\underline{\Lambda}(S) = \liminf_{r\to \infty}\frac{1}{\log r}\int_{S\cap(1,r)}\frac{dt}{t}.
\]
When $\overline{\Lambda}(S)=\underline{\Lambda}(S)$ we speak of the {\it logarithmic density} of $S$, denoted by $\Lambda(S)$.

Recall next that, for a continuous subharmonic function~$u$ in $\C$,
\[
A(r)=A(r,u)=\min_{|z|=r} u(z)\quad\text{and}\quad B(r)=B(r,u)=\max_{|z|=r} u(z),
\]
and the {\it order} and {\it lower order} of $u$ are
\[
\rho(u)=\limsup_{r\to\infty}\frac{\log B(r)}{\log r}\quad \text{and}\quad \lambda(u)=\liminf_{r\to\infty}\frac{\log B(r)}{\log r},
\]
respectively. For all $u\in \K$ we have $0\le \lambda(u) \le \rho(u)\le 1/2$.

The following result gives some basic properties of all functions in the class $\K$; see \cite[Theorem~1.2]{RS20}.
\begin{theorem}\label{basic-props}
Let $u\in \K$, with $E$ the corresponding closed subset of the negative real axis. Then $u$ has the following properties.
\begin{itemize}
\item[(a)]
Monotonicity properties:
for all $r>0$,
\[
u(re^{i\theta})\;\;\text{is decreasing as a function of } \theta, \text{ for } 0\le \theta \le \pi,
\]
so, in particular, $B(r,u)=u(r)$ and $A(r,u)=u(-r)$ for all $r>0$. Also,
\[
\frac{u(r)}{r^{1/2}}=\frac{B(r,u)}{r^{1/2}}\;\;\text{is decreasing for } r>0,
\]
so, in particular, $\rho(u)\le 1/2$.
\item[(b)] Bounds for order and lower order:
\[
\rho(u)\ge \tfrac12 \overline{\Lambda}(E^*)\quad {\text and} \quad \lambda(u)\ge \tfrac12 \underline{\Lambda}(E^*),
\]
where $E^*=\{x:-x\in E\}$.
\end{itemize}
\end{theorem}

The next result, again taken from~\cite[Theorem~1.5]{RS20}, concerns the behaviour of certain functions in the class $\K$.
\begin{theorem}\label{min-type}
Let $u\in \K$, with $E$ the corresponding closed subset of the negative real axis. If
\[E^c\supset\bigcup_{n\ge 0}(-d_n,-c_n),\]
where $0\le c_0<d_0<c_1<d_1< \cdots,$ and $\limsup_{n\to\infty} d_{n}/c_n>1$, then
\[
\frac{u(r)}{r^{1/2}}\to 0\;\;\text{as }r\to\infty.
\]
\end{theorem}
The final result we need from~\cite[Theorem~1.6]{RS20} describes the approximation used in the second stage of Kjellberg's process, showing how we can approximate a function $u\in\K$ by $\log |f|$, where $f$ is entire. This generalises the result given by Kjellberg for a particular type of set~$E$; see \cite[Chapter~4]{bK48}.
\begin{theorem}\label{discretise}
Suppose that $u\in\K$ and
\[
E=\bigcup_{n\ge 0}[-b_n,-a_n],
\]
where $0\le a_0<b_0<a_1<b_1< \cdots,$ and $a_n \to \infty$ as $n\to\infty$. Put
\[
D_1=\C\setminus \{z: \text{{\rm dist}}(z,E)\le 1\}.
\]
Then there exists an entire function~$f$ with only negative zeros, all lying in the set~$E$, such that
\begin{equation}\label{R-est}
\log|f(z)|-u(z)=O\left(\log|z|\right) \;\;\text{as }z\to \infty, \qfor z\in D_1.
\end{equation}
Moreover, if we also have
\begin{equation}\label{d-cond}
b_n/a_n \ge d >1,\qfor n\ge 0,
\end{equation}
then there exists $R=R(u)>0$ such that
\begin{equation}\label{f-upper-est}
\log|f(z)|\le u(z)+ 4\log |z|, \qfor |z|\ge R.
\end{equation}
\end{theorem}
Theorem~\ref{discretise} will enable us to approximate subharmonic functions $u \in\K$ by functions of the form $\log |f|$, where $f$ is a {\tef} of the same order, lower order and type class as~$u$.

\section{Proofs of Examples~\ref{main5} and \ref{main4}}
\label{orderhalf-negative}
\setcounter{equation}{0}

We prove Examples~\ref{main5} and~\ref{main4} by using Theorems~\ref{basic-props}, \ref{min-type} and \ref{discretise} to construct entire functions that approximate suitable subharmonic functions. In each case we show that property~\eqref{minmodprop} is false by arranging that the minimum modulus of the entire function is relatively small on intervals that are relatively long, consistent with the required growth of the maximum modulus.

We begin by proving the following result, needed in the construction of Example~\ref{main5}. The proof is rather long as the set $E$ has to be chosen carefully so that the function $u\in \K$ does not grow too quickly on the positive real axis and takes large values on the negative real axis only relatively rarely.

\begin{lemma}\label{example-u}
Let~$\delta :(0,\infty) \mapsto (0,1/4)$ be a decreasing function such that
\[
\delta(r)\to 0\;\text{ as } r \to \infty,
\]
and let $a_n$ and $b_n$, $n\ge 0$, be chosen to satisfy
\begin{equation}\label{est-u1}
a_{n+1}\ge b_n^{10}, \qfor n\ge 0,
\end{equation}
\begin{equation}\label{est-u2}
\delta(\tfrac12a_{n+1})\le \frac{\log b_n}{40b_n}, \qfor n\ge 0,
\end{equation}
and
\begin{equation}\label{est-u3}
\log b_{n+1}=\frac{b_n}{\log b_n}\log a_{n+1}, \qfor n\ge 0.
\end{equation}
Then the unique subharmonic function~$u\in\K$ corresponding to the set $E=\bigcup_{n\ge 0}[-b_n,-a_n]$, with $u(1)=1$, satisfies
\[
\rho(u)=1/2\quad\text{and }\quad u(r)/r^{1/2}\to 0\;\text{ as }r\to\infty,
\]
and also
\begin{itemize}
\item[(a)]
there exists $r_0>0$ such that
\begin{equation}\label{delta-est}
u(r) \le r^{1/2-\delta(r)}, \qfor r\ge r_0;
\end{equation}
\item[(b)]
we have $u(-r)=0$, for $r\in [a_n,b_n]$, $n\ge 0$, and, for~$n$ sufficiently large,
\begin{equation}\label{u(-r)}
u(-r) <  \log b_{n+1}, \qfor b_n\le r \le a_{n+1}.
\end{equation}
\end{itemize}
\end{lemma}

\begin{proof}
Let~$u\in\K$ be the unique subharmonic function with $u(1)=1$ corresponding to the set $E=\bigcup_{n\ge 0}[-b_n,-a_n]$, where $a_n$ and $b_n$ satisfy \eqref{est-u1}, \eqref{est-u2} and \eqref{est-u3}. The property \eqref{est-u3} implies that $\overline{\Lambda}(E^*)=1$, so~$u$ has order~1/2 by Theorem~\ref{basic-props}, part~(b).

Also, by \eqref{est-u1} and Theorem~\ref{min-type}, we have
\[
\frac{u(r)}{r^{1/2}}\to 0 \;\text{ as } r\to \infty.
\]
In the proof, we often use the following property of~$u$, from Theorem~\ref{basic-props}, part~(a):
\begin{equation}\label{uprop}
\frac{u(r)}{r^{1/2}} = \frac{B(r,u)}{r^{1/2}} \;\;\text{is decreasing as } r\to\infty.
\end{equation}
We also need further properties of $u(r)$ and of the average
\[
I(r)=\frac{1}{2\pi}\int_0^{2\pi} u(re^{i\theta})\,d\theta,\quad r>0,
\]
namely that
\begin{equation}\label{uIprop}
u(r)\text{ and } I(r) \text{ are positive increasing convex functions of } \log r,
\end{equation}
and
\begin{equation}\label{Iprop}
I(r)= A_n\log r +B_n, \qfor r\in [b_n,a_{n+1}], \;n\ge 0,
\end{equation}
for some constants $A_n>0$ and $B_n$, $n\ge 0$; see \cite[Section~2.7]{HK76}. Also,
\begin{equation}\label{uandI}
u(-r) <I(r)< u(r), \qfor r>0,
\end{equation}
by Theorem~\ref{basic-props}, part~(a), and
\begin{equation}\label{uIclose}
u(r) \le 3I(2r),\qfor r\ge 0,
\end{equation}
by using the Poisson integral of $u$ to majorise~$u$ in $\{z:|z|\le 2r\}$; see \cite[Theorem~2.5]{HK76} for example.

Taken together, these properties will give us good control over the behaviour of $u(r)$ and $u(-r)$ in terms of the sequences $(a_n)$ and $(b_n)$.

We write
\[
u(r)= r^{1/2-\eps(r)},\;\; r>0,
\]
where $0\le \eps(r)<1/2$, for $r>1$, by \eqref{uprop}. In order to verify \eqref{delta-est}, we need to show that $\eps(r)\ge \delta(r)$ for~$r$ sufficiently large.

To do this, we obtain bounds on the coefficients~$A_n$ and~$B_n$, $n\ge 0$, in~\eqref{Iprop}. First, we obtain an upper bound for~$A_n$. By \eqref{uprop}, \eqref{Iprop} and \eqref{uandI},
\[
\frac{A_n\log (b_n^2)+B_n}{b_n}=\frac{I(b_n^2)}{b_n} < \frac{u(b_n^2)}{b_n}\le \frac{u(b_n)}{b_n^{1/2}},\qfor n\ge 0,
\]
so
\begin{equation}\label{Anlogbn}
A_n\log b_n+(A_n\log b_n+B_n)< u(b_n)b_n^{1/2}\le b_n,\qfor n\ge 0,
\end{equation}
by \eqref{uprop} again. Since $A_n\log b_n+B_n=I(b_n)> 0$, we deduce from \eqref{Anlogbn} that
\begin{equation}\label{Bnbound}
A_n < \frac{b_n}{\log b_n},\qfor n\ge 0.
\end{equation}

Next we claim that $B_n<0$ for~$n$ sufficiently large. We do this by using the fact that $\phi(t)=I(e^t)$ is a positive increasing convex function with the property that $\phi(t)/t\to \infty$ as $t\to\infty$. This last property holds because~$\phi$ is convex and
\[
\limsup_{t\to\infty} \frac{\log \phi(t)}{t}=\limsup_{r\to\infty} \frac{\log I(r)}{\log r}=\limsup_{r\to\infty} \frac{\log u(r)}{\log r}=1/2,
\]
in view of \eqref{uIclose} and the fact that~$u$ has order $1/2$.

Now, by \eqref{Iprop},
\[
\frac{\phi(t)}{t}=A_n+\frac{B_n}{t},\;\text{ for } \log b_n \le t\le \log a_{n+1},
\]
and, by the convexity of $\phi$, we have
\[
A_n\ge \frac{\phi(\log b_n)-\phi(1)}{\log b_n - 1} > \frac{\phi(\log b_n)}{\log b_n}= A_n+\frac{B_n}{\log b_n},
\]
for~$n$ sufficiently large. Hence there exists a positive integer $N_1$ such that
\begin{equation}\label{Bn-neg}
B_n<0, \;\text{ for } n\ge N_1,
\end{equation}
as claimed.

The estimate \eqref{Bnbound} implies that there exists $N_2\ge N_1$ such that
\[
A_n < \frac{b_n}{\log b_n}\le \frac{r^{1/5}}{\log r^{1/5}}< \frac{r^{1/4}}{3\log 2r}\,,\qfor  r\ge b_n^5, n\ge N_2.
\]
Therefore, we deduce from \eqref{uIclose}, \eqref{Bn-neg} and \eqref{Iprop} that
\[
\frac{1}{r^{\eps(r)}}=\frac{u(r)}{r^{1/2}}\le \frac{3I(2r)}{r^{1/2}}< \frac{3A_n\log 2r}{r^{1/2}} < \frac{1}{r^{1/4}}, \qfor b_n^5\le r\le \tfrac12a_{n+1}, n\ge N_2.
\]
Hence
\begin{equation}\label{eps-large}
\eps(r) > \frac14 \ge \delta(r),\qfor b_n^5\le r \le \tfrac12a_{n+1}, n\ge N_2.
\end{equation}

To complete the proof of \eqref{delta-est}, we show that $\eps(r)\ge \delta(r)$ for $r\in [\frac12a_n,b_n^5]$ and $n\ge N_2+1$ by using the fact that
\[
\frac{u(r)}{r^{1/2}}\le \frac{u(\tfrac12a_n)}{(\tfrac12a_n)^{1/2}}, \qfor r\ge \tfrac12a_n, n\ge 0,
\]
by \eqref{uprop} once again, that is,
\[
\eps(r)\ge \frac{\eps(\tfrac12a_n)\log \tfrac12a_n}{\log r}, \qfor r \ge \tfrac12a_n, n\ge 0.
\]

Therefore, for $r\in [\frac12a_n,b_n^5]$ and $n\ge N_2+1$, we have, by \eqref{eps-large}, \eqref{est-u3} and \eqref{est-u2},
\begin{align}\label{eps-small}
\eps(r)&\ge \frac{\eps(\tfrac12a_n)\log \tfrac12a_n}{\log r}>\frac{\log \tfrac12a_n}{4\log r}\\
& \ge \frac{\log a_n^{1/2}}{20\log b_n}=\frac{\log b_{n-1}}{40b_{n-1}}\notag\\
&\ge \delta(\tfrac12a_n) \ge \delta(r),\notag
\end{align}
since $\delta$ is a decreasing function. Combining \eqref{eps-large} and \eqref{eps-small}, we obtain \eqref{delta-est}.

Finally, \eqref{u(-r)} follows immediately from the fact that, for $r\in [a_n,b_n]$, we have $u(-r)=0$ and for $r\in [b_n,a_{n+1}]$, $n \ge N_1$,  we have
\begin{align*}
u(-r)<I(r)&=A_n\log r+B_n\\
&\le A_n\log a_{n+1}\\
&\le \frac{b_n}{\log b_n} \log a_{n+1}\\
&=\log b_{n+1},
\end{align*}
by \eqref{uandI}, \eqref{Bn-neg}, \eqref{Bnbound} and \eqref{est-u3}.
\end{proof}

We now give the proof of Example~\ref{main5}. Here we use again the fact, mentioned in Section~\ref{orderhalf-positive}, that property~\eqref{minmodprop} holds if and only if
\begin{equation}\label{mtequiv1}
\text{there exists } R > 0 \text{ such that } \mt(r) >r, \text{ for } r \geq R,
\end{equation}
where
\[
\mt(r) := \max\{ m(s):0 \leq s \leq r\}, \qfor r \in [0, \infty).
\]
\begin{proof}[Proof of Example~\ref{main5}]
Without loss of generality we can assume that $\delta(r): (0,\infty) \to (0, 1/4)$ is a decreasing function such that $\delta(r)\to 0$ as $r \to \infty$.

Let~$u$ be the subharmonic function constructed in Lemma~\ref{example-u} and let $D=\C\setminus E=\C\setminus \bigcup_{n\ge 0}[-b_n,-a_n]$. We apply Theorem~\ref{discretise} to the function~$u$ to obtain an entire function~$f_1$ with zeros in the set~$E$ such that
\begin{equation}\label{fu1}
\log|f_1(z)|-u(z) = O(\log|z|)\;\text{ as }z\to\infty,\;z\in D_1,
\end{equation}
where $D_1=D\setminus \{z:{\rm dist}(z,\partial D)\le 1\}$, and also such that
\begin{equation}\label{fu2}
\log|f_1(z)| \le u(z)+ 4\log|z|, \qfor|z|\ge R,
\end{equation}
for some $R=R(u)>0$.

In view of the symmetry of~$u$ in the real axis, we can assume that $\arg f_1(x)=0$ for $x>0$.

We deduce from \eqref{fu1} and \eqref{fu2} that~$f_1$ has order $1/2$ minimal type, since~$u$ has these properties. In particular we can represent~$f_1$ as an infinite product of the form
\begin{equation}\label{prod}
f_1(z)=c\prod_{n=0}^{\infty} \left(1+\frac{z}{t_n}\right),\;\;\text{where } c>0 \text{ and } t_n\in E,\; n\ge 0.
\end{equation}
We also deduce from \eqref{fu2} that there exist $r_0>0$ such that
\begin{equation}\label{min1}
m(r,f_1)\le \exp (u(-r)+4\log r),\qfor r\ge r_0,
\end{equation}
and hence, since $u$ vanishes on the set $E$,
\begin{equation}\label{min2}
m(r,f_1)\le r^4,\qfor r\ge r_0, \; -r\in E.
\end{equation}

Now we consider the function~$f$ defined by
\begin{equation}\label{prodf}
f(z)=c\prod_{n=5}^{\infty} \left(1+\frac{z}{t_n}\right),
\end{equation}
and prove that~$f$ satisfies \eqref{max}. For some constant $K>0$, we have, by \eqref{fu2} and \eqref{delta-est},
\begin{align}
M(r,f)&\le M(r,f_1)\frac{K}{r^5}\notag\\
&\le \exp(u(r)+4\log r)\frac{K}{r^{5}}\notag\\
& \le \exp (u(r))\notag\\
&\le \exp\left(r^{1/2-\delta(r)}\right),\notag
\end{align}
provided that~$r$ is sufficiently large.

Finally, we show that property \eqref{mtequiv} does not hold for this function~$f$. For $r\in [a_{n+1}, b_{n+1}]$, we have
\[
m(r,f)\le m(r,f_1)\frac{K}{r^{5}}<1\le b_{n+1},
\]
provided that~$n$ is sufficiently large, by \eqref{prodf} and \eqref{min2}. Next, for $r\in [b_n,a_{n+1}]$, we have, by \eqref{fu2} and \eqref{u(-r)},
\begin{align}
m(r,f)&\le m(r,f_1)\frac{K}{r^{5}}\notag\\
&\le \exp\left(u(-r)+4\log r\right)\frac{K}{r^{5}}\notag\\
& \le \exp (u(-r))\le b_{n+1},\notag
\end{align}
provided that~$n$ is sufficiently large.

Therefore, for~$n$ sufficiently large,
\[
\mt(b_{n+1},f) = \max\{m(r):0\le r\le b_{n+1}\} \le b_{n+1},
\]
so property \eqref{mtequiv1} does not hold. Hence, for the function~$f$, there is no value of~$r$ such that $m^n(r)\to \infty$ as $n\to \infty$.

This completes the proof of Example~\ref{main5}.
\end{proof}

\begin{proof}[Proof of Example~\ref{main4}]
Here we follow a similar procedure to the previous proof, starting this time with a subharmonic function~$u\in \K$, where $E=\bigcup_{n\ge 0}[-b_n,-a_n]$, such that $u(1)=1$, $b_n/a_n\nearrow \infty$ as $n\to\infty$ and $a_{n+1}=2b_n$ for $n\ge 0$.

By Theorems~\ref{basic-props} and \ref{min-type}, the function~$u$ has order 1/2 minimal type and lower order 1/2, so the entire functions~$f_1$ and~$f$ resulting from applying Theorem~\ref{discretise} (as in the previous proof) also have these properties, by~\eqref{fu1} and~\eqref{fu2}.

Finally, to show that property~\eqref{mtequiv1} fails, we need to choose the sequences $(a_n)$ and $(b_n)$ iteratively so that
\begin{equation}\label{u(-r)again}
u(-r) <  \log b_{n+1}, \qfor b_n\le r \le a_{n+1}, n\ge 0,
\end{equation}
also holds, in order that
\[
m(r,f) \le \exp (u(-r))\le b_{n+1}, \qfor b_n\le r \le a_{n+1}, n\ge 0.
\]
Arranging for \eqref{u(-r)again} to hold is clearly possible since
\[
\max\{u(-r):b_n\le r\le a_{n+1}\} \le \max\{u(r):b_n\le r\le a_{n+1}\} \le a^{1/2}_{n+1},
\]
by Theorem~\ref{basic-props}, part~(a).
\end{proof}
{\it Acknowledgement}\; Phil Rippon expresses his deep gratitude to Walter Hayman, his PhD adviser, for Walter's inspiring research and teaching over many years. In particular, several aspects of the work in this paper can be traced to an early supervision at Imperial College when Walter passed on a copy of Bo Kjellberg's thesis \cite{bK48}, which contains the origins of the ideas that we use for constructing our examples and also the key Lemma~\ref{Beur}, due to Beurling, that we use to prove our positive theorems.

\end{document}